\documentclass[12pt,a4paper,psamsfonts]{amsart}
\usepackage{fancyhdr}
\usepackage{appendix}
\usepackage{amssymb,amscd,amsxtra,calc}
\usepackage{mathrsfs}
\usepackage{cmmib57}
\usepackage{multirow}
\usepackage[all]{xy}
\usepackage[colorlinks,linkcolor=blue,anchorcolor=blue,citecolor=green]{hyperref}
\theoremstyle{plain}
    \newtheorem{thm}{Theorem}[section]

     \newtheorem{conjecture}[thm]{Conjecture}
    \newtheorem{corollary}[thm]{Corollary}
    \newtheorem{example}[thm]{Example}
    \newtheorem{lemma}[thm]{Lemma}
    \newtheorem{proposition}[thm]{Proposition}
    \newtheorem{question}[thm]{Question}
    \newtheorem{theorem}[thm]{Theorem}

\theoremstyle{definition}
    \newtheorem{definition}[thm]{Definition}

\theoremstyle{remark}

\newcommand{\C}{\mathbb{C}}

\newcommand{\Q}{\mathbb{Q}}

\newcommand{\R}{\mathbb{R}}

\newcommand{\Z}{\mathbb{Z}}

\newcommand{\SEnd}{\operatorname{SEnd}}

\newcommand{\Amp}{\operatorname{Amp}}

\newcommand{\Aut}{\operatorname{Aut}}

\newcommand{\diag}{\operatorname{diag}}

\newcommand{\Gal}{\operatorname{Gal}}

\newcommand{\IEnd}{\operatorname{IEnd}}
\newcommand{\id}{\operatorname{id}}

\newcommand{\IAmp}{\operatorname{IAmp}}
\newcommand{\Ker}{\operatorname{Ker}}
\newcommand{\NE}{\overline{\operatorname{NE}}}
\newcommand{\Nef}{\operatorname{Nef}}
\newcommand{\NS}{\operatorname{NS}}

\newcommand{\PE}{\operatorname{PE}}

\newcommand{\PEnd}{\operatorname{PEnd}}
\newcommand{\Pol}{\operatorname{Pol}}

\newcommand{\N}{\operatorname{N}}

\newcommand{\Pic}{\operatorname{Pic}}

\newcommand{\ratmap}
{{\,\cdot\negmedspace\cdot\negmedspace\cdot\negmedspace\to\,}}

\newcommand{\alg}{\mathrm{alg}}
\newcommand{\reg}{\mathrm{reg}}

\begin{document}

\title[Polarized or int-amplified Endomorphisms]
{Normal projective varieties admitting polarized or int-amplified endomorphisms
}

\author{Sheng Meng and De-Qi Zhang}

\address
{
\textsc{Department of Mathematics} \endgraf
\textsc{National University of Singapore,
Singapore 119076, Republic of Singapore
}}
\email{math1103@outlook.com}
\email{ms@u.nus.edu}
\address
{
\textsc{Department of Mathematics} \endgraf
\textsc{National University of Singapore,
Singapore 119076, Republic of Singapore
}}
\email{matzdq@nus.edu.sg}

\begin{abstract}
Let $X$ be a normal projective variety admitting a polarized or int-amplified endomorphism $f$.
We list up characteristic properties of such an endomorphism and classify
such a variety from the aspects of its singularity,
anti-canonical divisor and Kodaira dimension.
Then we run the equivariant minimal model program with respect to not just the single $f$ but also 
the monoid $\SEnd(X)$
of all surjective endomorphisms of $X$,
up to finite-index.
Several applications are given.
We also give both algebraic and geometric characterizations of toric varieties via polarized endomorphisms.
\end{abstract}



\maketitle

\tableofcontents

\section{Introduction}
In this note, we report our recent results in studying surjective endomorphisms, especially polarized endomorphisms and int-amplified endomorphisms of higher dimensional algebraic varieties in arbitrary characteristic. The main focus is in setting up the equivariant minimal model program (MMP) for such endomorphisms. We will outline the ideas but refer to the original papers for the detailed proofs.
Our approach is more geometric.

For simplicity of the presentation, 
{\it we mainly work over an algebraically closed field $k$ of characteristic $0$}, 
except the last section where the cases of positive characteristic are discussed.

In Section \ref{sec-2}, we introduce various notions, like polarized endomorphisms, (int-) amplified endomorphisms, etc., and state the important properties enjoyed by these endomorphisms.

In Section \ref{sec-3}, we characterize normal projective varieties $X$ admitting amplified endomorphisms from the aspects of their singularities, canonical divisors, and Kodaira dimensions.
These are rough descriptions which will be used later for the theory of the equivariant MMP on $X$. 

In Section \ref{sec-4}, we show the equivariance of the MMP for a single polarized or int-amplified endomorphism.

In Section \ref{sec-5}, we show the finiteness of the number of contractible extremal rays for a projective variety $X$ admitting a polarized or int-amplified endomorphism.
This way, we confirm the equivariance of the MMP on $X$ for
a finite-index submonoid of the
monoid $\SEnd(X)$ of all surjective endomorphisms of $X$.

In Section \ref{sec-6}, we give the characterizations of toric varieties in terms of the existence of  almost homogeneity or polarized endomorphisms equipped with an extra ramification condition.

In Section \ref{sec-7}, we list several results which have appeared in the previous sections and also hold true in the case of positive characteristic. Of course, we need to assume some extra natural conditions, like separability of the map, etc.

We refer the readers to the survey paper \cite{Zhsw} which has more number-theoretic flavours and includes many outstanding conjectures.

\par \vskip 1pc
{\bf Acknowledgement.}
The second named author would like to thank the organizing committee for the kind invitation and
warm hospitality during the
International conference : Nevanlinna theory and Complex Geometry in Honor of Le Van Thiem's Centenary,
February - March, 2018, Hanoi, Vietnam.
Both authors would like to thank the referee for the very careful reading and the suggestions to improve and clarify the paper.
The first named author is supported by a Research Assistantship of NUS.
The second named author is supported by an Academic Research Fund of NUS.

\section{Polarized or int-amplified endomorphisms}
\label{sec-2}

Let $X$ be a projective variety.
A Cartier divisor is assumed to be integral, unless otherwise indicated.
Denote by $\Pic(X)$ the group of Cartier divisors modulo linear equivalence and $\Pic^0(X)$ the subgroup of the classes in $\Pic(X)$ which are algebraically equivalent to $0$.
Then
$$\NS(X) := \Pic(X)/\Pic^0(X)$$ is the N\'eron-Severi group.
Denote by
$$\N^1(X):=\NS(X) \otimes_{\Z} \mathbb{R}$$
and
$$\NS_{\mathbb{K}}(X):=\NS(X)\otimes_{\mathbb{Z}} \mathbb{K}$$ for $\mathbb{K} := \mathbb{Q}$, $\mathbb{R}$, or $\C$. So $\NS_{\R}(X) = \N^1(X)$.

Let $n:=\dim(X)$.
We can regard $\N^1(X)$ as the space of numerically equivalent classes of $\R$-Cartier divisors.
Two $\R$-Cartier divisors $D_1$ and $D_2$ are {\it numerically equivalent}, denoted as
$$D_1 \equiv D_2$$ if
their classes $[D_1]$ and $[D_2]$ in $\N^1(X)$ are the same, i.e., if
$$(D_1 - D_2) \cdot C = 0$$ for any curve $C$ on $X$.
Denote by
$$\N_r(X)$$ the space of weakly numerically equivalent classes of $r$-cycles with $\R$-coefficients.
Namely, two $r$-cylces $Z_1$ and $Z_2$ are {\it weakly numerically equivalent},
if
the intersection
$${(Z_1 - Z_2)}\cdot L_1 \cdots L_{n-r} = 0$$ for all Cartier divisors $L_i$ (cf.~\cite[Definition 2.2]{MZ}).
When $X$ is normal, we also call $\N_{n-1}(X)$ the space of weakly numerically equivalent classes of Weil $\R$-divisors.
In this case, $\N^1(X)$ can be regarded as a subspace of $\N_{n-1}(X)$
(cf.~\cite[Lemma 3.2]{Zh-TAMS}).

We recall the definitions of the following cones:
\begin{itemize}
\item $\Amp(X)$ is the cone of ample classes in $\N^1(X)$.
\item $\Nef(X)$ is the cone of nef classes in $\N^1(X)$.
\item $\PE^1(X)$ is the cone of pseudo-effective $\R$-Cartier divisor classes in $\N^1(X)$.
\item $\PE_{n-1}(X)$ is the cone of pseudo-effective Weil $\R$-divisor classes in $\N_{n-1}(X)$.
\end{itemize}
A Weil $\R$-divisor $B$ is said to be {\it big} if $B$ is in the interior part of $\PE_{n-1}(X)$.

Let $f:X\to X$ be a surjective endomorphism.
Then $f$ is a finite morphism of degree $d$.
We may define pullback of cycles for $f$, such that $f^\ast$ induces an automorphism of $\N_r(X)$ and $f_\ast f^\ast=d\,\id$; see \cite[Section 2.3]{Zh-comp}.
Note that the above cones are $f^\ast$-invariant.

We refer to \cite[\S2]{MZ} for more information.

\begin{definition}\label{def-f}
Let $f:X\to X$ be a surjective endomorphism of a projective variety $X$.
\begin{itemize}
\item[(1)]
$f$ is {\it numerically polarized}  if $f^{\ast}L \equiv qL$ for some ample Cartier divisor $L$ and integer $q>1$.
\item[(2)] $f$ is {\it numerically quasi-polarized}  if $f^{\ast}L \equiv qL$ for some big Cartier divisor $L$ and integer $q>1$.
\item[(3)] $f$ is {\it quasi-polarized} if $f^{\ast}L \sim qL$ for some big Cartier divisor $L$ and integer $q>1$.
\item[(4)] $f$ is $q$-{\it polarized}  if $f^{\ast}L \sim qL$ for some ample Cartier divisor $L$ and integer $q>1$.
\item[(5)] $f$ is {\it int-amplified} if $f^*L-L=H$ for some ample Cartier divisors $L$ and $H$.
\item[(6)] $f$ is {\it amplified} if $f^*L-L=H$ for some (not necessarily ample) Cartier divisor $L$ and some ample Cartier divisor $H$.
\end{itemize}
\end{definition}

For convenience, sometimes we simply say $f$ is polarized if it is $q$-polarized for some integer $q>1$.
It is easy to see that ``polarized'' implies ``int-amplified'' and ``int-amplified'' implies ``amplified''.
The notion of an amplified endomorphism $f : X \to X$ was first introduced by Krieger and Reschke (cf.~\cite{Kr}); preceding to this, Fakhruddin \cite[Theorem 5.1]{Fak} has shown that for such $f$, the set of $f$-periodic points is Zarisiki dense in $X$.
There do exist amplified {\it automorphisms} (e.g., automorphisms of positive entropy on abelian surfaces), while the degree of an int-amplified endomorphism is always greater than $1$ (cf.~\cite[Lemma 3.10]{Meng}).
Unlike the polarized property (cf.~\cite[Corollary 3.12]{MZ}),
in general, it is impossible to preserve the amplified automorphism property via a birational equivariant lifting (cf.~\cite[Lemma 4.4]{Kr} and \cite[Theorem 1.2]{Re}).

On one hand, we do not know whether the amplified property can be descended via an equivariant morphism (cf.~\cite[Question 1.10]{Kr}).
On the other hand, int-amplified endomorphisms have all the nice properties enjoyed by the polarized endomorphisms (cf.~\cite[\S 3]{MZ}).

Next we focus on showing the equivalence of the first four polarizations in Definition \ref{def-f}.
The following is a norm criterion for a numerically polarized endomorphism.
\begin{proposition}\label{prop-fx-x}(cf.~\cite[Proposition 2.9]{MZ}, \cite[Proposition 3.1]{CMZ})
Let $\varphi:V\to V$ be an invertible linear map of a positive dimensional real vector space equipped with a norm.
Assume
$$\varphi(C)=C$$ for a convex cone $C\subseteq V$ such that $C$ spans $V$ and its closure $\overline{C}$ contains no line.
Let $q$ be a positive number. Then the conditions (i) and (ii) below are equivalent.
\begin{itemize}
\item[(i)] $\varphi(u)=q u$ for some $u\in C^\circ$ (the interior part of $C$).
\item[(ii)]
There exists a constant $N>0$, such that
$$\frac{||\varphi^i||}{q^i}< N$$ for all $i\in \mathbb{Z}$.
\end{itemize}
Assume further the equivalent conditions (i) and (ii). Then the following are true.
\begin{itemize}
\item[(1)] $\varphi$ is a diagonalizable linear map with all eigenvalues of modulus $q$.
\item[(2)] Suppose $q>1$. Then, for any $v\in V$ such that $\varphi(v)-v\in C$, we have $v\in C$.
\end{itemize}
\end{proposition}

Applying the above criterion to the cones $\Nef(X)$ and $\PE^1(X)$, we now can say the equivalence of ``numerically quasi-polarized'' and ``numerically polarized''; see \cite[Proposition 3.6]{MZ}.
Furthermore, applying \cite[Lemma 2.3]{Na-Zh}, ``numerically polarized'' is equivalent to ``polarized'' by taking $H$ there to be ample.

\begin{theorem}\label{thm-num-lin}(cf.~\cite[Proposition 1.1]{MZ})
Let $f:X\to X$ be a numerically quasi-polarized endomorphism of a projective variety $X$. Then $f$ is polarized.
\end{theorem}

The following are useful criteria for int-amplified endomorphisms.

\begin{proposition}\label{prop-3-equiv}(cf.~\cite[Proposition 3.3]{Meng}) Let $f:X\to X$ be a surjective endomorphism of a projective variety $X$.
Then the following are equivalent.
\begin{itemize}
\item[(1)] The endomorphism $f$ is int-amplified.
\item[(2)] All the eigenvalues of
$$\varphi:=f^*|_{\N^1(X)}$$ are of modulus greater than $1$.
\item[(3)] There exists some big $\R$-Cartier divisor $B$ such that $f^*B-B$ is big.
\item[(4)] If $C$ is a $\varphi$-invariant convex cone in $\N^1(X)$,
then
$$\emptyset\neq(\varphi-\id_{\N^1(X)})^{-1}(C)\subseteq C .$$
\end{itemize}
\end{proposition}

Considering the action $f^*|_{\N_{n-1}(X)}$ and the cone $\PE_{n-1}(X)$, we have similar criteria as follows.

\begin{proposition}\label{prop-4-equiv}(cf.~\cite[Proposition 3.4]{Meng}) Let $f:X\to X$ be a surjective endomorphism of an $n$-dimensional normal projective variety $X$.
Then the following are equivalent.
\begin{itemize}
\item[(1)] The endomorphism $f$ is int-amplified.
\item[(2)] All the eigenvalues of
$$\varphi:=f^*|_{\N_{n-1}(X)}$$ are of modulus greater than $1$.
\item[(3)] There exists some big Weil $\R$-divisor $B$ such that $f^*B-B$ is a big Weil $\R$-divisor.
\item[(4)] If $C$ is a $\varphi$-invariant convex cone in $\N_{n-1}(X)$,
then
$$\emptyset\neq(\varphi-\id_{\N_{n-1}(X)})^{-1}(C)\subseteq C .$$
\end{itemize}
\end{proposition}

By the above criteria, we may easily show that the properties of ``numerically polarized'' and ``int-amplified'' are preserved via every equivariant descending.

\begin{proposition}\label{lem-int-des1}(cf.~\cite[Theorem 3.11]{MZ}, \cite[Lemmas 3.5, 3.6]{Meng}) Let $\pi:X\dasharrow Y$ be a dominant rational map of projective varieties.
Let $f:X\to X$ and $g:Y\to Y$ be two surjective endomorphisms such that 
$$g\circ\pi=\pi\circ f .$$
Suppose $f$ is numerically polarized (resp.~int-amplified). Then $g$ is also numerically polarized (resp.~int-amplified).
\end{proposition}

\section{Singularities, anti-canonical divisor and Kodaira dimension}\label{sec-3}

We refer to \cite[Chapters 2 and 5]{KM} for the definitions and the properties of log canonical (lc), Kawamata log terminal (klt), canonical and terminal singularities.

Let $f:X\to X$ be a surjective endomorphism of a normal projective variety $X$.
When $\dim(X)=2$, Wahl \cite[Theorem 2.8]{Wa} showed that
$X$ has at worst lc singularities.
Broustet and H\"oring \cite[Corollary 1.5]{BH} generalized this result to higher dimensions by adding extra assumptions that $f$ is polarized and $X$ is $\Q$-Gorenstein.
We may further weaken the extra assumption ``$f$ is polarized" to the assumption ``$f$ is int-amplified."

\begin{theorem}\label{main-thm-lc}(cf.~\cite[Theorem 1.6]{Meng})
Let $X$ be a $\Q$-Gorenstein normal projective variety over the field $k$ of characteristic $0$ admitting an int-amplified endomorphism.
Then $X$ has at worst lc singularities.
\end{theorem}

When $f$ is polarized and $X$ is smooth, Boucksom, de Fernex and Favre \cite[Theorem C]{BFF} showed that $-K_X$ is pseudo-effective.
Cascini, Meng and Zhang \cite[Theorem 1.1 and Remark 3.2]{CMZ} used a different method to
show further that $-K_X$ is weakly numerically equivalent to an effective Weil $\Q$-divisor without the assumption of $X$ being smooth.
By applying Propositions \ref{prop-3-equiv} and \ref{prop-4-equiv} and the ramification divisor formula, we have the following result for the int-amplified case.

\begin{theorem}\label{main-thm-k}(cf.~\cite[Theorem 1.5]{Meng}) Let $X$ be a normal projective variety admitting an int-amplified endomorphism.
Then $-K_X$ is weakly numerically equivalent to some effective Weil $\Q$-divisor.
If $X$ is further assumed to be $\Q$-Gorenstein, then $-K_X$ is numerically equivalent to some effective $\Q$-Cartier divisor.
\end{theorem}

Let $f:X\to X$ be an amplified endomorphism of a projective variety $X$.
By taking the equivariant Iitaka fibration (cf.~\cite[Theorem A]{NZ09}), we obtain the following constraint on the Kodaira dimension of $X$.

\begin{theorem}\label{lem-amp-kod}(cf.~\cite[Lemma 2.5]{Meng})
Let $f:X\to X$ be an amplified endomorphism of a projective variety $X$.
Then the Kodaira dimension $\kappa(X)\le 0$.
\end{theorem}

\section{Equivariant MMP with respect to an endomorphism}\label{sec-4}

This section generalises and extends results in \cite{Zh-comp} to higher dimensions.

Let
$$(*):X_1\dasharrow X_2\dasharrow\cdots\dasharrow X_r$$
be a finite sequence of dominant rational maps of projective varieties.
Let $f:X_1\to X_1$ be a surjective endomorphism.
We say the sequence $(*)$ is {\it $f$-equivariant} if the following diagram is commutative
$$\xymatrix{
X_1\ar@{-->}[r]\ar[d]^{f_1} &X_2\ar@{-->}[r]\ar[d]^{f_2} &\cdots\ar@{-->}[r] &X_r\ar[d]^{f_r}\\
X_1\ar@{-->}[r] &X_2\ar@{-->}[r] &\cdots\ar@{-->}[r] &X_r\\
}
$$
where $f_1=f$ and all $f_i$ are surjective endomorphisms.
Let $S$ be a set of surjective endomorphisms of $X$.
We say the sequence $(*)$ is {\it $S$-equivariant} if it is $g$-equivariant for any $g\in S$.

We first need the following key lemma for the theory of the equivariant MMP.

\begin{lemma}\label{lem-fin-per}(cf.~\cite[Lemma 8.1]{Meng}, \cite[Lemma 6.1]{MZ})
Let $f:X\to X$  be an int-amplified endomorphism of a projective variety $X$.
Assume $A\subseteq X$ is a closed subvariety with 
$$f^{-i}f^i(A) = A$$ 
for all $i\ge 0$.
Then $A$ is $f^{-1}$-periodic (i.e., $f^{-s}(A)=A$ for some $s>0$).
\end{lemma}

Applying the same argument and proofs of \cite[Lemma 6.2 to Lemma 6.6]{MZ}, we obtain the following theorem.
Note that Proposition \ref{lem-int-des1} and Theorem \ref{thm-num-lin} are needed to show the following $g$ is again polarized or int-amplified.

\begin{theorem}\label{thm-equi-mmp}(cf.~\cite[Theorem 8.2]{Meng}) Let $f:X\to X$ be a polarized (resp.~int-amplified) endomorphism of a $\Q$-factorial lc projective variety $X$.
Let $\pi:X\dasharrow Y$ be a dominant rational map which is 
\begin{itemize}
\item[(i)] 
either a divisorial contraction, or 
\item[(ii)]
a Fano contraction, or 
\item[(iii)]
a flipping contraction, or 
\item[(iv)]
a flip
\end{itemize}
induced by a $K_X$-negative extremal ray.
Then there exists a polarized (resp.~int-amplified) endomorphism $g:Y\to Y$ such that 
$$g\circ\pi=\pi\circ f$$ 
after replacing $f$ by a positive power.
\end{theorem}

Now assuming the existence of an MMP, we have the following result of equivariant MMP by successively applying Theorem \ref{thm-equi-mmp}.
We will also give a stronger version of the following theorem in the next section (cf. Theorem \ref{main-thm-finite-R}).

\begin{theorem}\label{thm-emmp}
Let $f:X\to X$ be a polarized (or int-amplified) endomorphism of a $\Q$-factorial lc projective variety $X$.
Then any finite sequence of MMP starting from $X$ is $f^s$-equivariant for some $s>0$.
\end{theorem}

The existence of a polarized or int-amplified endomorphism on a variety
exerts strong constraints on the geometry of the variety.
Recall that a normal projective variety $X$ is said to be {\it $Q$-abelian} if there is a finite surjective morphism $\pi:A\to X$ \'etale in codimension $1$ with $A$ being an abelian variety.

\begin{theorem}\label{main-thm-qa}
(cf. \cite[Theorem 3.4]{Na-Zh}, \cite[Theorem 1.21]{GKP}, \cite[Lemma 6.9]{MZ}, \cite[Theorem 1.9]{Meng})
Let $f:X\to X$ be an int-amplified endomorphism of a normal projective variety $X$.
Suppose either $X$ is klt and $K_X$ is pseudo-effective or $X$ is non-uniruled.
Then $X$ is $Q$-abelian.
\end{theorem}

We refer to \cite{KM}, \cite{BCHM} and \cite{Fu15} for details about an MMP.
Together with the previous theorem, one may characterize step by step the equivariant MMP and the action $f^*|_{\N^1(X)}$.
\begin{theorem}\label{main-thm-mmp}(cf.~\cite[Theorem 1.8]{MZ}, \cite[Theorem 1.10]{Meng})
Let $f:X\to X$ be an int-amplified endomorphism of a $\mathbb{Q}$-factorial klt projective variety $X$.
Then, replacing $f$ by a positive power, there exists a finite sequence $(*)$ of $f$-equivariant MMP as at the beginning of the section,
where we set $X_1 := X$, the $X_i \ratmap X_{i+1}$ is a divisorial contraction, flip, or Fano-contraction corresponding to a $K_{X_i}$-negative extremal ray, 
and $Y = X_r$ is a $Q$-abelian variety, such that the following assertions hold.
\begin{itemize}
\item[(1)]
If $K_X$ is pseudo-effective, then $X=Y$ and it is $Q$-abelian.
\item[(2)]
If $K_X$ is not pseudo-effective, then for each $i$, $X_i\to Y$ is an equi-dimensional well-defined morphism with every fibre irreducible (and rationally connected if the base field $k$ is further uncountable) and $f_i$ is int-amplified. The $X_{r-1}\to X_r = Y$ is a Fano contraction.
\item[(3)]
$f^*|_{\N^1(X)}$ is diagonalizable over $\mathbb{C}$ if and only if so  is $f_{r}^\ast|_{\N^1(Y)}$.
\item[(4)]
If $f$ is polarized, then $f^*|_{\N^1(X)}$ is a scalar multiplication:
$$f^*|_{\N^1(X)} = q \, \id,$$ if and only if so  is $f_{r}^\ast|_{\N^1(Y)}$.
\end{itemize}
\end{theorem}

\begin{definition}\label{def-q} Let $X$ be a normal projective variety.
\begin{itemize}
\item[(1)]
$q(X):=h^1(X,\mathcal{O}_X)=\dim H^1(X,\mathcal{O}_X)$ (the irregularity).
\item[(2)]
$\tilde{q}(X):=q(\tilde{X})$ with $\tilde{X}$ a smooth projective model of $X$.
\item[(3)]
$q^\natural(X):=\sup\{\tilde{q}(X')\,|\,X'\to X \text{ is finite surjective and \'etale in codimension one}\}$.
\item[(4)]
$\pi_1^{\alg}(X_{\reg})$ is the algebraic fundamental group of the smooth locus $X_{\reg}$ of $X$.
\end{itemize}
\end{definition}

The following result is an application of Theorem \ref{main-thm-mmp}. Note that when $f:X\to X$ is a polarized endomorphism of a projective variety $X$, the action $f^*|_{\N^1_{\C}(X)}$ is always diagonalizable and all the eigenvalues are of the same modulus (cf.~\cite[Proposition 2.9]{MZ}).
However, without the extra assumptions on $X$ as stated, Theorem \ref{main-thm-rc-diag} (2) fails, for a general int-amplified endomorphism;
see Example \ref{exa-fak} given by Najmuddin Fakhruddin.

\begin{theorem}\label{main-thm-rc-diag}(cf.~\cite[Lemma 9.1, Theorem 1.10]{MZ}, \cite[Theorem 1.11]{Meng}) Let $f:X\to X$ be an int-amplified endomorphism of a $\Q$-factorial klt projective variety $X$.
Suppose either $q^\natural(X)=0$ or $\pi_1^{\alg}(X_{\reg})$ is finite (e.g., $X$ is smooth and rationally connected).
Then we have:
\begin{itemize}
\item[(1)] There exists a finite sequence of MMP which ends up with a point.
\item[(2)] There exists some $s>0$, such that $(f^s)^*|_{\N^1(X)}$ is diagonalizable over $\Q$ with all the eigenvalues being positive integers greater than $1$.
\item[(3)] If $f$ is further $q$-polarized, then $(f^s)^*|_{\N^1(X)}=q^s \, \id$.
\end{itemize}
\end{theorem}

\begin{example} {\rm (Fakhruddin)} \label{exa-fak}
\rm{} Let $E$ be an elliptic curve admitting a complex multiplication.
Let
$$S=E\times E .$$
Then $\dim (\N^1(S))=4$.
Let $\sigma:S\to S$ be an automorphism given as:
$$(x,y)\mapsto (x,x+y) .$$
Then $\sigma^*|_{\N^1(S)}$ is not diagonalizable over $\mathbb{C}$.
Let $m_S$ be the multiplication endomorpphism of $S$.
Note that
$$m_S^*|_{\N^1(S)}=m^2\id_{\N^1(S)} .$$
So $f:=\sigma\circ m_S$ is int-amplified for $m > 1$, by Proposition \ref{prop-3-equiv}.
Clearly,
$$f^*|_{\N^1(S)}$$ is not diagonalizable over $\C$.
\end{example}

\section{Equivariant MMP with respect to the monoid $\SEnd(X)$}\label{sec-5}

We use Proposition \ref{prop-finiteclosed} below in proving the results in the rest of this section.
As kindly informed by Professors Dinh and Sibony, this kind of result (with a complete proof) first appeared
in \cite[Section 3.4]{DS03}; \cite[Theorem 3.2]{Dinh09} is a more general form including Proposition \ref{prop-finiteclosed} below, requiring a weaker condition and dealing with also dominant meromorphic self-maps of K\"ahler manifolds;
see comments in \cite[page 615]{DS10} for the history of these results; see also \cite{BD}.

\begin{proposition}\label{prop-finiteclosed}(cf.~\cite[Proposition 3.6]{MZ_PG})
Let $f:X\to X$  be an int-amplified endomorphism of a projective variety $X$.
Then there are only finitely many $f^{-1}$-periodic Zariski closed subsets.
\end{proposition}

Let $X$ be a projective variety and let $C$ be a curve.
Denote by 
$$R_C:=\mathbb{R}_{\ge 0}[C]$$ 
the ray generated by $[C]$ in $\NE(X)$.
Denote by 
$$\Sigma_C$$ 
the union of curves whose classes are in $R_C$.

\begin{definition}\label{def:extrem_ray} Let $X$ be a projective variety.
Let $C$ be a curve such that $R_C$ is an extremal ray in $\NE(X)$.
We say $C$ or $R_C$ is {\it contractible} if there is a surjective morphism
$\pi : X \to Y$ to a projective variety $Y$ such that
the following hold.

\begin{itemize}
\item[(1)] $\pi_*\mathcal{O}_X=\mathcal{O}_Y$.
\item[(2)] Let $C'$ be a curve in $X$.
Then $\pi(C')$ is a point if and only if $[C'] \in R_C$.
\item[(3)] Let $D$ be a $\Q$-Cartier divisor of $X$.
Then $D\cdot C=0$ if and only if $D\equiv \pi^*D_Y$ (numerical equivalence) for some $\Q$-Cartier divisor $D_Y$ of $Y$.
\end{itemize}
\end{definition}

A submonoid $G$ of a monoid $\Gamma$ is said to be of {\it finite-index} in $\Gamma$
if there is a chain 
$$G = G_0 \le G_1 \le \cdots \le G_r = \Gamma$$ 
of submonoids and homomorphisms $\rho_i : G_i \to F_i$ such that $\Ker(\rho_i) = G_{i-1}$ and all $F_i$ are finite
{\it groups}.

Denote by 
$$\SEnd(X)$$ 
the monoid of all surjective endomorphisms of $X$.

\begin{theorem}\label{main-thm-finite-R} (cf.~\cite[Theorem 1.1]{MZ_PG})
Let $X$ be a (not necessarily normal or $\Q$-Gorenstein) projective variety with a polarized (or int-amplified) endomorphism $f$. Then:
\begin{itemize}
\item[(1)]
$X$ has only finitely many (not necessarily $K_X$-negative) contractible extremal rays in the sense of Definition \ref{def:extrem_ray}.
\item[(2)]
Suppose that $X$ is $\Q$-factorial and normal. Then any finite sequence of MMP starting from $X$ is $G$-equivariant for some finite-index submonoid $G$ of $\SEnd(X)$.
\end{itemize}
\end{theorem}

We give the main idea of the proof of Theorem \ref{main-thm-finite-R}.
First, we make use of Lemma \ref{lem-fin-per} to show that the exceptional locus of each contraction is $f^{-1}$-periodic.
By Proposition \ref{prop-finiteclosed}, there are only finitely many $f^{-1}$-periodic closed subvarieties.
Finally, we show that there are only finitely many contractible extremal rays sharing the same exceptional locus.
 
\begin{theorem}\label{main-thm-GMMP}(cf.~\cite[Theorem 1.2]{MZ_PG})
Let $f:X\to X$ be an int-amplified endomorphism of a $\mathbb{Q}$-factorial klt projective variety $X$.
Then there exist a finite-index submonoid $G$ of $\SEnd(X)$,
a $Q$-abelian variety $Y$, and a $G$-equivariant relative MMP over $Y$
$$X=X_0 \dashrightarrow \cdots \dashrightarrow X_i \dashrightarrow \cdots \dashrightarrow X_r=Y$$
(i.e. $g \in G = G_0$ descends to $g_i \in G_i$ on each $X_i$), such that the following assertions hold.
\begin{itemize}
\item[(1)]
There is a finite quasi-\'etale (i.e., \'etale in 
codimension $1$) Galois cover $A \to Y$ from an abelian variety $A$
such that $G_Y := G_r$ lifts to a submonoid $G_A$ of $\SEnd(A)$.
\item[(2)]
If $g$ in $G$ is amplified and
its descending $g_i$ on $X_i$ is int-amplified for some $i$, then $g$ is int-amplified.
The $X_{r-1}\to X_r = Y$ is a Fano contraction.
\item[(3)]
For any subset $H \subseteq G$ and its descending 
$$H_Y \subseteq \SEnd(Y)$$
$H$ acts via pullback on $\NS_{\Q}(X)$ or $\NS_{\mathbb{C}}(X)$ as commutative diagonal matrices with respect
to a suitable basis if and only if so does $H_Y$.
\end{itemize}
\end{theorem}

Let 
$$\PEnd(X)$$ 
be the set of {\it all polarized endomorphisms} on $X$, and
let 
$$\IEnd(X)$$ 
be the set of {\it all int-amplified endomorphisms} on $X$.
In our earlier papers, we used $\Pol(X)$ and $\IAmp(X)$ 
instead of $\PEnd(X)$ and $\IEnd(X)$, respectively.
The new notation here is suggested by the referee.

In general, these sets are not semigroups, i.e., they may not be closed under composition; see \cite[Example 10.4]{Meng}.
When $X$ is rationally connected and smooth,
Theorem \ref{main-thm-rc} below gives the assertion that if $g$ and $h$ are in $\PEnd(X)$ (resp.~$\IEnd(X)$)
then $g^M \circ h^M$ remains in $\PEnd(X)$ (resp.~ $\IEnd(X)$) for some $M>0$ depending only on $X$.
In particular, it answers affirmatively \cite[Question 4.15]{YZ}, ``up to finite-index", when $X$ is rationally connected and smooth.
By \cite[Example 1.7]{MZ_PG}, this extra ``up to finite-index" assumption
is necessary.

\begin{theorem}\label{main-thm-rc} (cf.~\cite[Theorems 1.4, 6.2]{MZ_PG})
Let $X$ be a rationally connected smooth projective variety admitting a polarized (or int-amplified) endomorphism $f$.
Then there are a finite-index submonoid $G \le \SEnd(X)$ and an integer $M > 0$ both depending only on $X$,
such that:
\begin{itemize}
\item[(1)]
$G^*|_{\NS_{\Q}(X)}$ is a commutative diagonal
monoid with respect to a suitable $\Q$-basis $B$ of $\NS_{\Q}(X)$.
Further, for every $g$ in $G$,
the representation matrix $[g^*|_{\NS_{\Q}(X)}]_B$ relative to $B$,
is equal to $$\diag[q_1, q_2, \dots]$$
with integers $q_i \ge 1$.
\item[(2)]
$G \cap \PEnd(X)$ is a subsemigroup of $G$, and consists exactly of those $g$ in $G$
such that 
$$[g^*|_{\NS_{\Q}(X)}]_B = \diag[q, \dots, q]$$ 
for some integer $q \ge 2$.
\item[(3)]
$G \cap \IEnd(X)$ is a subsemigroup of $G$, and consists exactly of those $g$ in $G$
such that 
$$[g^*|_{\NS_{\Q}(X)}]_B = \diag[q_1, q_2, \dots]$$ 
with integers $q_i \ge 2$.
\item[(4)]
We have $h^M\in G$ and that $h^*|_{\NS_{\C}(X)}$ is diagonalizable for every $h \in \SEnd(X)$.
\end{itemize}
\end{theorem}

Let $\Aut(X)$ be the group of all automorphisms of $X$, and $\Aut_0(X)$ its neutral connected component.
By applying Theorems \ref{main-thm-GMMP} and \ref{main-thm-rc}, we have the following result.

\begin{theorem}\label{main-thm-auto}(cf.~\cite[Theorems 1.5, 6.3]{MZ_PG})
Let $X$ be a rationally connected smooth projective variety.
Suppose $X$ admits a polarized (or int-amplified) endomorphism.
Then we have:
\begin{itemize}
\item[(1)]
$\Aut(X)/\Aut_0(X)$ is a finite group. More precisely, $\Aut(X)$ is a linear algebraic group
(with only finitely many connected components).
\item[(2)]
Every amplified endomorphism of $X$ is int-amplified.
\item[(3)]
$X$ has no automorphism of positive entropy (nor amplified automorphism).
\end{itemize}
\end{theorem}

\section{Characterizations of toric varieties}\label{sec-6}

A normal projective variety $X$ is said to be {\it toric} or a {\it toric variety}
if $X$ contains an algebraic torus $T = (k^*)^n$ as an (affine) open dense subset such that
the natural multiplication action of $T$ on itself extends to an action on the whole variety $X$.
In this case, let 
$$D:=X\backslash T$$ 
which is a divisor; the pair $(X,D)$ is said to be a {\it toric pair}.

Let $G$ be a linear algebraic group acting on a normal projective variety $X$.
We say $X$ is {\it $G$-almost homogeneous} if there exists an open dense $G$-orbit in $X$.
Note that $X$ is toric if and only if $X$ is $T$-almost homogeneous for some algebraic torus $T$.
The following result gives a sufficient condition, in terms of a polarized endomorphism,
for an almost homogeneous variety to be toric.
The key ingredient in the proof is the main result in
\cite{BZ}.

\begin{theorem}\label{thm-torus}(cf.~\cite[Theorem 1.1]{MZ-toric})
Let $f:X\to X$ be a polarized endomorphism of a $G$-almost homogeneous normal projective variety $X$ with $G$ being a linear algebraic group.
Assume further the following conditions.
\begin{itemize}
\item[(i)] Let $U$ be the open dense $G$-orbit in $X$ and $D$ the codimension-$1$ part of $X \setminus U$.
The Weil divisor $K_X+D$ is $\Q$-Cartier.
\item[(ii)]
The endomorphism $f$ is $G$-equivariant in the sense: there is a surjective
homomorphism $\varphi: G \to G$ such that
$$f\circ g=\varphi(g)\circ f$$ 
for all $g$ in $G$.
\end{itemize}
Then $(X, D)$ is a toric pair.
\end{theorem}

Next we are going to state a geometric characterization of a toric pair.
Let $(X,\Delta)$ be a log pair, i.e.,
$X$ is a normal projective variety, $\Delta$ is an effective $\Q$-divisor and the divisor $K_X + \Delta$ is $\Q$-Cartier.
The {\it complexity} 
$$c = c(X,\Delta)$$ 
of the pair $(X,\Delta)$ is defined as
$$c:=\inf\{n+\dim_{\R}(\sum_i \, \R[S_i])-\sum a_i\,|\, \sum a_i S_i\le \Delta, a_i\ge 0, S_i\ge 0\}.$$
Here $\sum_i \, \R[S_i]$ is a subspace, spanned by the Weil divisor classes $[S_i]$, of the real space spanned by all Weil divisor classes on $X$ (modulo algebraic equivalence).
Brown, ${\rm M}^c$Kernan, Svaldi and Zong recently gave a geometric characterization of toric varieties involving the complexity; see \cite[Theorem 1.2]{BMSZ}.
Their result is a special case of a conjecture of Shokurov, which is stated in the relative case (cf.~\cite{Sho}).
A simple version of their result shows that
{\it if $(X,\Delta)$ is a log canonical pair such that $\Delta$ is reduced, $-(K_X + \Delta)$ is nef, and
$c(X,\Delta)$ is non-positive, then $(X,\Delta)$ is a toric pair};
see \cite[Remark 4.4]{MZ-toric}.

Let $X$ be an $n$-dimensional smooth Fano variety of Picard number one
and $D\subset X$ a reduced divisor.
Assume the existence of a non-isomorphic surjective endomorphism $f : X \to X$ such that  $D$ is $f^{-1}$-invariant and $f|_{X\backslash D}$ is \'etale.
Hwang and Nakayama have shown that $X$ is isomorphic to $\mathbb{P}^n$ and $D$ is a simple normal crossing divisor consisting of $n + 1$ hyperplanes;
see \cite[Theorem 2.1]{HN}.
In particular, $(X,D)$ is a toric pair.
Indeed, their argument shows that the complexity $c(X,D)$ is non-positive.

Our Theorem \ref{thm-pitrivial} follows their idea and tries to generalize their result to the singular case. Precisely, let $X$ be a normal projective variety and $D$ a reduced divisor such that the pair 
$(X, D)$ is log smooth over an open set $U \subseteq X$, 
i.e., $U$ is smooth and $D \cap U$ is of simple normal crossing,
where $X \setminus U$ has codimension $\ge 2$ in $X$.
Let 
$$\hat{\Omega}_X^1(\log D) = \iota_*\hat{\Omega}_X^1(\log D)|_U$$
with $\iota: U \to X$ the inclusion, which is a reflexive sheaf and is independent of
the choice of the above $U$.
Our key step is to verify that 
$$\hat{\Omega}_X^1(\log D)$$
is free, i.e., isomorphic to $\mathcal{O}_X^{\oplus n}$;
see \cite[Proposition 2.3]{HN} and \cite[Theorem 5.4]{MZ-toric}.

\begin{theorem}\label{thm-pitrivial}(cf.~\cite[Theorem 1.2]{MZ-toric})
Let $X$ be a normal projective variety
which is smooth in codimension $2$, and $D\subset X$ a reduced divisor such that
\begin{itemize}
\item[(i)]
there is a Weil $\Q$-divisor $\Gamma$ such that the pair $(X, \Gamma)$ has only klt singularities;
\item[(ii)]
there is a polarized endomorphism $f:X\to X$ such that $D$ is $f^{-1}$-invariant and $f|_{X\backslash D}$ is quasi-\'etale;
\item[(iii)]
the algebraic fundamental group
$\pi_1^{\alg}(X_{\reg})$ of the smooth locus $X_{\reg}$ of $X$ is trivial
(this holds when $X$ is smooth and rationally connected); and
\item[(iv)]
the irregularity $q(X) := h^1(X, \mathcal{O}_X)$ is zero (this holds when $X$ is rationally connected).
\end{itemize}
Then the complexity $c(X,D)$ is non-positive.
\end{theorem}

An immediate corollary is the following.
\begin{corollary}\label{smooth-cor}(cf.~\cite[Corollary 1.4]{MZ-toric}) Let $X$ be a rationally connected smooth projective variety and $D\subset X$ a reduced divisor.
Suppose $f:X\to X$ is a polarized endomorphism such that $D$ is $f^{-1}$-invariant and $f|_{X\backslash D}$ is quasi-\'etale.
Then $(X,D)$ is a toric pair.
\end{corollary}

We say that a normal projective variety $X$ is of {\it Fano type} if there is a Weil $\Q$-divisor $\Delta$ such that the pair $(X,\Delta)$ has only klt singularities and $-(K_X+\Delta)$ is an ample $\Q$-Cartier divisor.
The assumption below of $X$ being of Fano type is necessary,
since a normal projective toric variety is known to be of Fano type.
The lifting in the following corollary is usually needed too; see \cite[Remark 1.7]{MZ-toric}.

\begin{corollary}\label{toric-lift-cor}(cf.~\cite[Corollary 1.5]{MZ-toric})
Let $f:X\to X$ be a polarized endomorphism of a normal projective variety $X$ of Fano type which is smooth in codimension $2$.
Let $D\subset X$ be an $f^{-1}$-invariant reduced divisor such that $f|_{X\backslash D}$ is quasi-\'etale and $K_X+D$ is $\Q$-Cartier.
Then there exist a quasi-\'etale cover 
$$\pi:\widetilde{X}\to X$$ 
and a polarized endomorphism 
$$\widetilde{f}:\widetilde{X}\to \widetilde{X}$$ 
such that
\begin{itemize}
\item[(1)] the endomorphism $f$ lifts to $\tilde{f}$, i.e., 
$$\pi\circ \widetilde{f}=f\circ \pi ,$$
and
\item[(2)] the pair $(\widetilde{X},\widetilde{D})$ is toric, where $\widetilde{D}=\pi^{-1}(D)$.
\end{itemize}
\end{corollary}

The following well known conjecture is still open (see also Question \ref{que-toric}).
It has been affirmatively solved in dimension $\le 3$.
We refer to \cite{ARV}, \cite{Be}, \cite{CL}, \cite{HM}, \cite{HN}, \cite{Na-Zh} and \cite{PS} for details.

\begin{conjecture}
Let $X$ be a Fano manifold of Picard number one which is different from the
projective space. Then a surjective endomorphism $X\to X$ must be bijective.
\end{conjecture}

Generalizing the above conjecture to arbitrary Picard number, we may ask:

\begin{question}\label{que-toric}
Let $X$ be a rationally connected smooth projective variety of dimension $n \ge 1$
which admits a polarized endomorphism.
Is $X$ (close to) a toric variety?
\end{question}

\section{Cases in positive characteristic}\label{sec-7}

Throughout this section, we always work over the field $k$ of characteristic $p>0$.

\begin{definition} Let $f:X\to X$ be a surjective endomorphism of a projective variety $X$.
We say that  $f$ is {\it separable} if the induced field extension $f^*:k(X)\to k(X)$ is separable.
Denote by $f^{\Gal}: Y\to X$ the Galois closure of $f$.
\end{definition}

Let $f:X\to X$ be a separable surjective endomorphism of a normal projective variety $X$. We have seen that $f$ being ``numerically quasi-polarized'' implies $f$ being ``numerically polarized''; see Proposition \ref{prop-fx-x} which is applied to $\N^1(X)$ and hence is characteristic free.
Furthermore, we have the following result generalizing Theorem \ref{thm-num-lin}.

\begin{theorem}\label{thm-num-lin-p}(cf.~\cite[Theorem 5.1]{CMZ})
Let $f:X\to X$ be a numerically quasi-polarized separable endomorphism of a normal projective variety $X$. Then $f$ is polarized.
\end{theorem}

Propositions \ref{prop-3-equiv} and \ref{lem-int-des1} hold true in the case of positive characteristic since the argument are cone-theoretical.
Let $X$ be an $n$-dimensional normal projective variety.
Then $\N^1(X)$ can be naturally embedded into $\N_{n-1}(X)$ due to the following lemma.
So Proposition \ref{prop-4-equiv} also holds true in the case of positive characteristic.
\begin{lemma}
Let $X$ be a projective variety of dimension $n \ge 2$. Let $M$ a Cartier divisor. Suppose that $H_1\cdots H_{n-1}\cdot M = 0 = H_1\cdots H_{n-2}\cdot M^2$ for some ample Cartier divisors $H_1, \cdots, H_{n-1}$. Then $M\equiv 0$ (numerical equivalence).
\end{lemma}
\begin{proof}
When $n=2$, this lemma follows from the Hodge index theorem.
We then prove by induction on $n$.
Suppose $n>2$.
By taking normalization, we may assume $X$ is normal.
Let $C$ be an irreducible curve on $X$.
Since $n>2$, there exists an integral hypersurface $H\in |mH_1|$ containing $C$ for some $m>0$ (cf.~\cite[Lemma 2.4]{Hu}).
Note that $H_2|_H\cdots H_{n-1}|_H\cdot M|_H = 0 = H_2|_H\cdots H_{n-2}|_H\cdot (M|_H)^2$.
By induction $M\cdot C=M|_H\cdot C=0$.
So the lemma is proved.
\end{proof}

With some additional assumptions and extra work, we can deal with the singularities for the surface case,
extending \cite{Wa} to positive characterisitics.
We refer to \cite[Lemma 4.4]{Ok} for the generalized ramification divisor formula. 
In the case of positive characteristic, one requires an additional ``tame'' assumption on $f$ to apply \cite[Proposition 5.20]{KM}; see \cite[Proposition 4.6, Remark 4.7, Example 4.8]{Ok}.

\begin{theorem}\label{thm-w}(cf.~\cite[Theorem 10.2]{CMZ}) Let $f:X\to X$ be a non-isomorphic surjective endomorphism of a normal algebraic surface $X$ such that the degree $\deg f^{\Gal}$ of the Galois closure of $f$ is co-prime to $p$.
Then $X$ is lc.
\end{theorem}

The proof of Theorem \ref{main-thm-k} is based on Propositions \ref{prop-3-equiv} and \ref{prop-4-equiv} and the ramification divisor formula. So we can easily get the following.

\begin{theorem}\label{main-thm-k-p}(cf.~\cite[Theorem 1.1]{CMZ}, \cite[Theorem 1.5]{Meng}) Let $X$ be a normal projective variety admitting a polarized (or int-amplified) separable endomorphism.
Then we have:
\begin{itemize}
\item[(1)]
$-K_X$ is weakly numerically equivalent to some effective Weil $\Q$-divisor.
\item[(2)]
If $X$ is further assumed to be $\Q$-Gorenstein, then $-K_X$ is numerically equivalent to some effective $\Q$-Cartier divisor.
\end{itemize}
\end{theorem}

By a pure algebraic proof, we further have:

\begin{theorem}\label{thm-torsion}(cf.~\cite[Theorem 1.4]{CMZ})
Let $f:X\to X$ be a numerically polarized separable endomorphism of a normal projective variety $X$ with $K_X$ being pseudo-effective and $\Q$-Cartier. Then $f$ is quasi-\'etale and $K_X\sim_{\Q}0$.
\end{theorem}

Following the proof of \cite[Lemma 6.1]{MZ}, \cite[Lemma 6.2]{CMZ} and \cite[Lemma 8.1]{Meng}, we may generalize Lemma \ref{lem-fin-per} as follows.
We refer to \cite[Remark 6.3]{CMZ} to see that the following condition (2) is necessary.
\begin{lemma}(cf.~\cite[Lemmas 3.4 and 3.5]{MZ_PG})
Let $f:X\to X$  be an int-amplified separable endomorphism of a projective variety $X$.
Assume $A\subset X$ is a closed subvariety with 
$$f^{-i}f^i(A) = A$$ 
for all $i\ge 0$.
Assume further either one of the following conditions.
\begin{itemize}
\item[(1)] $A$ is a prime divisor of $X$.
\item[(2)] $p$ and $\deg f$ are co-prime.
\end{itemize}
Then $A$ is $f^{-1}$-periodic.
\end{lemma}

The assumption ``$p$ and $\deg f$ are co-prime'' for the following proposition is also necessary; see \cite[Remark 6.3]{CMZ}.
\begin{proposition}\label{prop-finiteclosed-p}(cf.~\cite[Proposition 3.6]{MZ_PG})
Let $f:X\to X$  be an int-amplified endomorphism of a projective variety $X$.
Suppose $p$ and $\deg f$ are co-prime.
Then there are only finitely many $f^{-1}$-periodic Zariski closed subsets.
\end{proposition}

With the above two results, Theorem \ref{main-thm-finite-R} can also be generalized as follows.

\begin{theorem}\label{main-thm-finite-R-p} (cf.~\cite[Theorem 1.1]{MZ_PG})
Let $X$ be a (not necessarily normal or $\Q$-Gorenstein) projective variety with a polarized (or int-amplified) endomorphism. 
Suppose $p$ and $\deg f$ are co-prime.
Then:
\begin{itemize}
\item[(1)]
$X$ has only finitely many (not necessarily $K_X$-negative) contractible extremal rays in the sense of Definition \ref{def:extrem_ray}.
\item[(2)]
Suppose $X$ is $\Q$-factorial normal. Then any finite sequence of MMP starting from $X$ is $G$-equivariant for some finite-index submonoid $G$ of $\SEnd(X)$.
\end{itemize}
\end{theorem}

In the case of positive characteristic, the theory of MMP is still far from being completed and is only known for lc 3-folds with characteristic $p>5$ (cf.~\cite{BW}, \cite{HX}, \cite{Wal} and the references therein).
For the lower dimensional cases, we refer to \cite[Theorem 1.6]{CMZ} and \cite[Theorem 1.8]{CMZ} for versions similar to Theorems \ref{main-thm-mmp}  and \ref{main-thm-rc-diag}, respectively.


\begin{thebibliography}{99}
\bibitem{ARV}
E. Amerik, M. Rovinsky and A. Van de Ven,
A boundedness theorem for morphisms between threefolds,
Ann. Inst. Fourier (Grenoble) \textbf{49} (1999), 405-415.

\bibitem{Be}
A. Beauville,
Endomorphisms of hypersurfaces and other manifolds, Intern. Math. Res. Notices \textbf{2001}, no. 1 (2001): 53-58.

\bibitem{BCHM}
C. Birkar, P. Cascini, C. D. Hacon and J. McKernan,
Existence of minimal models for varieties of log general type, J. Amer. Math. Soc., \textbf{23}(2):405-468, 2010.

\bibitem{BW}
C. Birkar and J.~Waldron,
Existence of Mori fibre spaces for $3$-folds in ${\rm char}\,p$,
Adv. Math. \textbf{313} (2017), 62-101.


\bibitem{BFF}
S.~Boucksom, T.~de Fernex and C.~Favre,
The volume of an isolated singularity, Duke Math. J. \textbf{161} (2012), no. 8, 1455-1520.

\bibitem{BD}
J. -Y. Briend and J. Duval,
Erratum: Deux caractérisations de la mesure d'\'equilibre d'un endomorphisme de $P^k(\C)$,
Publ. Math. Inst. Hautes \'Etudes Sci. No. 109 (2009), 295-296.

\bibitem{BZ}
M.~Brion and D.-Q.~Zhang,
Log Kodaira dimension of homogeneous varieties,
Algebraic varieties and automorphism groups, 1-6, 
Adv. Stud. Pure Math., \textbf{75}, Math. Soc. Japan, Tokyo, 2017. 

\bibitem{BH}
A.~Broustet and A.~H\"oring,
Singularities of varieties admitting an endomorphism,
Math. Ann. \textbf{360} (2014), no. 1-2, 439-456.


\bibitem{BMSZ}
M.~Brown, J.~${\rm M}^c$Kernan, R.~Svaldi and H.~Zong,
A geometric characterisation of toric varieties,
Duke Math. J., Volume \textbf{167}, Number 5 (2018), 923-968.


\bibitem{CMZ}
P.~Cascini, S.~Meng and D.-Q.~Zhang,
Polarized endomorphisms of normal projective threefolds in arbitrary characteristic,
\href{http://arXiv.org/abs/1710.01903}{ arxiv:1710.01903}.

\bibitem{CL}
D. Cerveau and A. Lins Neto,
Hypersurfaces exceptionnelles des endomorphismes de CP(n), Bol.
Soc. Brasil. Mat. (N.S.) \textbf{31} (2000), no. 2, 155-161.

\bibitem{Dinh09}
T. -C. Dinh,
Analytic multiplicative cocycles over holomorphic dynamical systems,
Complex Var. Elliptic Equ. \textbf{54} (2009), no. 3-4, 243-251.

\bibitem{DS03}
T. -C. Dinh and N. Sibony,
Dynamique des applications d'allure polynomiale,
J. Math. Pures Appl. \textbf{82} (2003), pp. 367-423.

\bibitem{DS10}
T. -C. Dinh and N. Sibony,
Equidistribution speed for endomorphisms of projective spaces,
Math. Ann. \textbf{347} (2010), no. 3, 613-626.

\bibitem{Fak}
N.~Fakhruddin,
Questions on self-maps of algebraic varieties,
J. Ramanujan Math. Soc., \textbf{18}(2):109-122, 2003.

\bibitem{Fu15}
O.~Fujino, Some remarks on the minimal model program for log canonical pairs,
J. Math. Sci. Univ. Tokyo \textbf{22} (2015), no. 1, 149-192.

\bibitem{GKP}
D.~Greb, S.~Kebekus and T.~Peternell,
\'Etale fundamental groups of Kawamata log terminal spaces, flat sheaves, and quotients of Abelian varieties,
Duke Math. J. \textbf{165} (2016), no. 10, 1965-2004.

\bibitem{HX}
C.~Hacon and C.~Xu,
On the three dimensional minimal model program in positive characteristic,
J. Amer. Math. Soc. \textbf{28} (2015), no. 3, 711-744.

\bibitem{Hu}
F. Hu, 
A theorem of Tits type for automorphism groups of projective varieties in arbitrary characteristic,
\href{http://arXiv.org/abs/1801.06555}{ arxiv:1801.06555}.

\bibitem{HM}
J.-M. Hwang and N. Mok,
Finite morphisms onto Fano manifolds of Picard number 1 which
have rational curves with trivial normal bundles,
J. Alg. Geom. \textbf{12} (2003), 627-651.

\bibitem{HN}
J.~M.~Hwang and N.~Nakayama, On endomorphisms of Fano manifolds of Picard number one, Pure Appl. Math. Q., \textbf{7}(4), pp.1407-1426, 2011.


\bibitem{KM}
J.~Koll\'ar and S.~Mori,
Birational geometry of algebraic varieties,
Cambridge Tracts in Math.,
\textbf{134} Cambridge Univ. Press, 1998.

\bibitem{Kr}
H.~Krieger and P.~Reschke,
Cohomological conditions on endomorphisms of projective varieties,
Bull. Soc. Math. France \textbf{145} (2017), no. 3, 449-468.

\bibitem{Meng}
S.~Meng,
Building blocks of amplified endomorphisms of normal projective varieties,
\href{http://arXiv.org/abs/1712.08995}{ arXiv:1712.08995}.

\bibitem{MZ}
S.~Meng and D.-Q.~Zhang,
Building blocks of polarized endomorphisms of normal projective varieties,
Adv. Math. \textbf{325} (2018), 243-273.

\bibitem{MZ-toric}
S.~Meng and D.-Q.~Zhang,
Characterizations of toric varieties via polarized endomorphisms,
Math. Z. (to appear),
\href{http://arXiv.org/abs/1702.07883}{ arXiv:1702.07883}.

\bibitem{MZ_PG}
S.~Meng and D.-Q.~Zhang,
Semi-group structure of all endomorphisms of a projective variety admitting a polarized endomorphism,
Math. Res. Lett. (to appear), 2019,
\href{https://arxiv.org/abs/1806.05828}{ arXiv:1806.05828}.

\bibitem{NZ09}
N.~Nakayama and D.-Q.~Zhang,
Building blocks of \'etale endomorphisms of complex projective manifolds,
Proc. Lond. Math. Soc. \textbf{99} (2009), no. 3, 725-756.

\bibitem{Na-Zh}
N.~Nakayama and D.-Q.~Zhang,
Polarized endomorphisms of complex normal varieties,
Math. Ann. \textbf{346} (2010), no. 4, 991-1018.

\bibitem{Ok}
S.~Okawa,
Extensions of two Chow stability criteria to positive characteristics,
Michigan Math. J. \textbf{60} (2011), no. 3, 687-703.

\bibitem{PS}
K. H. Paranjape and V. Srinivas,
Self maps of homogeneous spaces,
Invent. Math. \textbf{98} (1989), 425-444.

\bibitem{Re}
P.~Reschke,
Distinguished line bundles for complex surface automorphisms,
Transform. Groups, \textbf{19}(1):225-246, 2014.

\bibitem{Sho}
V.~V.~Shokurov,
Complements on surfaces, J. Math. Sci. (New York) \textbf{102} (2000), no. 2,
3876-3932, Algebraic geometry, \textbf{10}.


\bibitem{Wal}
J.~Waldron,
The LMMP for log canonical 3-folds in char $p$,
\href{http://arXiv.org/abs/1603.02967}{ arXiv:1603.02967}.

\bibitem{Wa}
J.~Wahl,
A characteristic number for links of surface singularities,
J. Amer. Math. Soc. \textbf{3} (1990), no. 3, 625-637.

\bibitem{YZ}
X. Yuan and S. Zhang,
The arithmetic Hodge index theorem for adelic line bundles,
Math. Ann. \textbf{367} (2017), no. 3-4, 1123-1171.

\bibitem{Zh-comp}
D.-Q.~Zhang, Polarized endomorphisms of uniruled varieties, Compos. Math. \textbf{146} (2010), no. 1, 145-168.

\bibitem{Zh-TAMS}
D.-Q.~Zhang, $n$-dimensional projective varieties with the action of an abelian group of rank $n - 1$,
Trans. Amer. Math. Soc. \textbf{368} (2016), no. 12, 8849-8872.

\bibitem{Zhsw}
S.~W.~Zhang, “Distributions in algebraic dynamics, 381-430,” Survey in Differential
Geometry \textbf{10}, Somerville, MA, International Press, 2006.

%
%
%
%

\end{thebibliography}
\end{document}